\documentclass[reqno,12pt]{amsart} 
\usepackage{amsmath,amssymb,amsthm,enumerate,mathrsfs,colonequals}
\usepackage{fullpage,url,tikz,xspace,setspace}
\usepackage{microtype}
\usepackage{calc}

\usepackage[
]{hyperref}
\usepackage[alphabetic,lite,nobysame]{amsrefs} 

\def\be{\begin{equation}}
\def\ee{\end{equation}}

\def\0{^{\phantom 0}}
\def\9{_{\phantom q}}

\newtheorem{theorem}{Theorem}[section]
\newtheorem{lemma}[theorem]{Lemma}
\newtheorem{proposition}[theorem]{Proposition}
\theoremstyle{remark}
\newtheorem{remark1}[theorem]{Remark}
\theoremstyle{definition}
\newtheorem{definition1}[theorem]{Definition}

\newcommand{\A}{\mathbb A}
\newcommand{\V}{\mathbb V}
\newcommand{\Vv}{\mathscr V}

\newcommand{\Z}{\mathbb Z}
\newcommand{\oV}{\overline{V}\!}

\newcommand{\OO}{\mathcal O}
\newcommand{\UU}{\mathcal U}
\newcommand{\VV}{\mathcal V}

\definecolorset{rgb}{}{}%
{purpviolet,0.5,0,0.25}%

\DeclareMathOperator\Mat{Mat}

\DeclareMathOperator\GL{GL}
\DeclareMathOperator\GLO{\GL_{1}}
\DeclareMathOperator\SL{SL}
\DeclareMathOperator\SLT{\SL_{2}}
\begin{document}

\title{Integral points on varieties defined by matrix factorization into elementary matrices}


\author{Bruce~W.~Jordan}
\address{Department of Mathematics, Baruch College, The City University
of New York, One Bernard Baruch Way, New York, NY 10010-5526, USA}
\email{bruce.jordan@baruch.cuny.edu}

\author{Yevgeny~Zaytman}
\address{Center for Communications Research, 805 Bunn Drive, Princeton, 
NJ 08540-1966, USA}
\email{ykzaytm@idaccr.org}

\begin{abstract}
Let $\OO$ be the ring of $S$-integers in a number field $K$.
For $A\in \SLT(\OO)$ and $k\geq 1$, we define matrix-factorization
varieties $V_k(A)$ over $\OO$ which parametrize factoring $A$ into a 
product of $k$ elementary matrices; the equations defining $V_k(A)$ are
written in terms of Euler's continuant polynomials.  We show that the $V_k(A)$
are rational $(k-3)$-folds with an inductive fibration structure.  We combine this
geometric structure with arithmetic results to study the Zariski closure
of the $\OO$-points of $V_k(A)$.  We prove that for  $k\geq 4$ the $\OO$-points on
$V_k(A)$ are Zariski dense if $V_{k}(A)(\OO)\neq\emptyset$ assuming the group of units $\OO^{\times}$ is infinite.  This shows that if $A$ can be written
as a product of $k\geq 4$ elementary matrices, then this can be done
in infinitely many ways in the strongest sense possible.
This can then be combined with results on factoring into elementary matrices for $\SLT(\OO)$.
One result is that for $k\geq 9$ the $\OO$-points on $V_{k}(A)$ are Zariski dense if $\OO^{\times}$ is infinite.
\end{abstract}

\maketitle

\section{Introduction}
\label{rabbit}

Let $K$ be a number field and let $S$ be a finite set
of primes of $K$ containing the archimedean valuations.
Denote by $\OO=\OO_S$ the ring of $S$-integers in $K$:
\[
\OO=\OO_{S}=\{x\in K^{\times} \mid v(x)\geq 0 \text{ for all $v\notin S$}\}.
\]
Factoring $1$ in $\OO^\times=\GLO(\OO)$ is naturally parametrized by associated
factorization varieties.  For $k\geq 2$, set
\begin{equation}
\label{factor}
\VV_k=\VV_k(1): 1=x_1\cdots x_k .
\end{equation}
The (affine) variety $\VV_k$ is smooth, rational, irreducible, and of dimension $k-1$.
When $k=2$, $\VV_2$ is a curve of genus $0$  with $2$ points at infinity.
Because there are fewer than $3$ points at $\infty$ on this rational curve, Siegel's Theorem
\cite{siegel}, \cite[Chapter 7]{serre-mw} does not force $\VV_2$ to have finitely many $\OO$-points.
And, indeed, $\VV_2(\OO)$ is Zariski dense if and only if $\OO^{\times}$
is infinite.  It is easy to see that this generalizes for $k>2$:
\begin{theorem}
\label{viper}
The $\OO$-points on $\VV_k$ for $k\geq 2$
are Zariski dense if and only if the group of units $\OO^\times$ is infinite.
\end{theorem}

We would like to generalize this picture to matrices and consider integral points on
matrix-factorization varieties.  In this paper we give one 
such generalization to $\SLT(\OO)$, considering the varieties parametrizing
factoring into a product of elementary matrices.
For $x\in \OO$ define the upper triangular matrix $U(x)$,
the lower triangular matrix $L(x)$, and the matrix $D(x)$ by
\begin{equation}
  \label{elementary}
  U(x):= \begin{pmatrix} 1 & x\\0& 1\end{pmatrix},
    \quad L(x):=\begin{pmatrix} 1 & 0\\
    x & 1\end{pmatrix}, \quad D(x):=
\begin{pmatrix} x & 1\\1 & 0\end{pmatrix}.
\end{equation}
The {\em elementary matrices} over $\OO$ are the
matrices $U(x)$, $L(x)$ for $x\in\OO$; $\det D(x)=-1$.
Set $t:=\left(\begin{smallmatrix}0&1\\1&0\end{smallmatrix}\right)$
and note the identities
\begin{equation}
\label{identities}
D(x)t=U(x)\quad\mbox{and}\quad t D(x)=L(x).
\end{equation}

We consider the (affine) matrix-factorization varieties defined by writing 
$A\in\SLT(\OO)$
as a product of $k$ matrices from \eqref{elementary}:
\begin{definition1}
\label{products}
{\rm
Suppose $k\geq 1$.\\
1. Define the variety $V_{k}(A)$ by writing
$A$ a product of $k$ elementary matrices beginning with a \emph{lower}
triangular matrix:
\begin{equation}
\label{lower}
V_k(A):\quad A=
\begin{cases}
  L(x_1)U(x_2)L(x_3)U(x_4) \cdots L(x_k)\quad\mbox{if $k$ is odd}\\
  L(x_1)U(x_2)L(x_3)U(x_4)\cdots U(x_k)\quad\mbox{if $k$ is even.}
\end{cases}
\end{equation}
2.  Define the variety $V^k(A)$ by writing 
$A$ a product of $k$ elementary matrices beginning with an \emph{upper}
triangular matrix:
\begin{equation}
\label{upper}
V^k(A):\quad A=
\begin{cases}
  U(x_1)L(x_2)U(x_3)L(x_4) \cdots U(x_k)\quad\mbox{if $k$ is odd}\\
  U(x_1)L(x_2)U(x_3)L(x_4)\cdots L(x_k)\quad\mbox{if $k$ is even.}
\end{cases}
\end{equation}
3.  Put $\V_k(A)=V_k(A)\coprod V^k(A)$.  Note that $\V_k(A)(\OO)\neq
\emptyset$ if and only if $A$ can be written as the product of $k$
elementary matrices.\\[.05in]
4.  Define the variety $\oV_k(A)$ by writing $A$ as product of $k$ of the
$D(x)$'s and $t^k$:
\begin{equation}
\label{D}
\oV_k(A):\quad A=
D(x_1)D(x_2)D(x_3)\cdots D(x_k)t^k.
\end{equation}
}
\end{definition1}

In this paper our main result is the following analogue for matrix-factorization
varieties of Theorem \ref{viper}:
\begin{theorem}
\label{dense2}
Suppose $A\in\SLT(\OO)$ and let $\Vv_{k}(A)$ be 
$V_{k}(A)$, $V^{k}(A)$, or $\oV_{k}(A)$.
Assume the group of units $\OO^\times$ is infinite.
\\
\textup{1)} Suppose $k\geq 9$.  Then the $\OO$-points on the rational and irreducible 
$(k-3)$-fold $\Vv_k(A)$ are Zariski dense.
\\
\textup{2)} Suppose $S$ contains a real archimedean or a finite prime.  Then the $\OO$-points
on $\Vv_k(A)$ are Zariski dense for $k\geq 8$.
\\
\textup{3)} Assume the Generalized Riemann Hypothesis as in \cite[Section 4]{jz}. Then the $\OO$-points on $\Vv_k(A)$ are
Zariski dense for $k\geq 5$ if $S$ contains a real archimedean prime, 
for $k\geq 6$ if $S$
contains a finite prime, and for $k\geq 7$ in general.  
\end{theorem}

\noindent Roughly speaking, Theorem \ref{dense2} 
says that if $\#\OO^\times =\infty$, then for $k$ 
sufficiently large,
not only can every $A\in\SLT(\OO)$ be written as the product of $k$
elementary matrices over $\OO$, but this can be done in
(infinitely) many ways in the strongest possible sense.

\section{First properties of the Matrix-Factorization Varieties
$V_k(A)$, $V^k(A)$, and~$\oV_k(A)$}
\label{first}

In fact we only have to study one of the three matrix-factorization 
varieties in Definition \ref{products} and 
the three factorization problems are equivalent.
Let 
\begin{equation}
  \label{weather}
A:=   \begin{pmatrix} a & c \\ b& d\end{pmatrix}\in\SLT(\OO),\ 
A':=   \begin{pmatrix} d & b \\ c & a\end{pmatrix}\in\SLT(\OO),\ \mbox{and}\ 
A^{\ast}:= \begin{pmatrix} d & c \\ b & a\end{pmatrix}\in\SLT(\OO),
\end{equation}
with $A^{t}=\left(\begin{smallmatrix}a & b\\c&d\end{smallmatrix}\right)$.
Then we have $A''=A^{tt}=A^{\ast\ast}=A$ and 
\begin{equation}
\label{snow}
A^{\prime t}=A^\ast,\quad A^{t \ast}=A',\quad A^{\ast \prime}=A^t,\quad
A^{\prime \ast}=A^t, \quad A^{t \prime}=A^{\ast}, \quad A^{\ast t}=A'.
\end{equation}
In other words, $\{\prime , \,t, \ast= \prime t\}$ are the 
elements of order $2$ in a Klein Vierergruppe acting on the set $\SLT(\OO)$.

\begin{proposition}
\label{one}
\textup{1. }Suppose $A\in\SLT(\OO)$.
The polynomial equations in $x_1, \, x_2, \, \ldots, x_k$
defining $V_k(A)$, $V^k(A')$, and $\oV_k(A')$ in $\A^{k}_{/\OO}$
are identical. The following are equivalent:\\[.05in]
\textup{a) } $(a_1, a_2,\ldots , a_k)\in V_{k}(A)(\OO)\subseteq \A^{k}(\OO)$.\\
\textup{b) } $\begin{cases} A=L(a_1)U(a_2)\cdots L(a_k)\quad
\mbox{if $k$ is odd}\\
A=L(a_1)U(a_2)\cdots U(a_k)\quad\mbox{ if $k$ is even}.  \end{cases}$\\
\textup{c) }  $(a_1, a_2,\ldots , a_k)\in V^{k}(A')(\OO)\subseteq \A^{k}(\OO)$.\\
\textup{d) }  $\begin{cases} A'=U(a_1)L(a_2)\cdots U(a_k)\quad
\mbox{if $k$ is odd}\\
A'=U(a_1)L(a_2)\cdots L(a_k)\quad\mbox{if $k$ is even}.  \end{cases}$\\
\textup{e) } $(a_1, a_2,\ldots , a_k)\in \oV_{k}(A')(\OO)\subseteq \A^{k}(\OO)$\\
\textup{f) } $A'=D(a_1)D(a_2)\cdots D(a_k)t^k$.\\[.1in]
\textup{2. } In particular, the following are equivalent:\\[.05in]
\textup{a) } For every matrix $A\in\SLT(\OO)$, $V_{k}(A)(\OO)\neq\emptyset$.\\
\textup{b) } For every matrix $A\in\SLT(\OO)$, $V^{k}(A)(\OO)\neq\emptyset$.\\
\textup{c) } For every matrix $A\in\SLT(\OO)$, $\oV_{k}(A)(\OO)\neq\emptyset$.
\end{proposition}
\begin{proof}
1)  The varieties $\oV_k(A)$ and $V^k(A)$ are defined by the same 
equations.\\
If $k$ is even, then 
\begin{align*}
A &= D(x_1)D(x_2)D(x_3)D(x_4)\cdots D(x_k)\\
  &=D(x_1)ttD(x_2)D(x_3)ttD(x_4)\cdots D(x_{k-1})ttD(x_k)\\
  &= U(x_1)L(x_2)U(x_3)L(x_4)\cdots U(x_{k-1})L(x_k)
\quad\mbox{using \eqref{identities}}.
\end{align*}
If  $k$ is odd, then similarly
\begin{equation*}
A=D(x_1)\cdots D(x_k)t=U(x_1)L(x_2)U(x_3)L(x_4)\cdots U(x_k),
\end{equation*}
proving 1).\\
2)  The varieties $V_k(A)$ and $V^k(A')$ are defined by the 
same equations.\\
We have $L(x)'=U(x)$ and $U(x)'=L(x)$.  Verify by induction
that for $k$ odd 
\[
[L(x_1)U(x_2)L(x_3)\cdots L(x_k)]'=U(x_1)L(x_2)U(x_3)\cdots U(x_k)
\]
and for $k$ even
\[
[L(x_1)U(x_2)L(x_3)\cdots U(x_k)]'=U(x_1)L(x_2)U(x_3)\cdots L(x_k),
\]
proving 2).
\end{proof}

\begin{proposition}
\label{two}
Suppose $A\in\SLT(\OO)$. The following are equivalent:\\[.05in]
\textup{a) } $(a_1, a_2,\ldots , a_k)\in V_{k}(A)(\OO)\subseteq \A^{k}(\OO)$.\\
\textup{b) }  $\begin{cases} (a_k, a_{k-1}, \ldots , a_1)\in V_{k}(A^{\ast})(\OO)\subseteq \A^{k}(\OO)\quad
\mbox{if $k$ is odd}\\
(a_k,a_{k-1}, \ldots , a_1)\in V_{k}(A^t)(\OO)\subseteq \A^{k}(\OO)\quad\mbox{if $k$ is even}.  \end{cases}$
\end{proposition}
\begin{proof}
Proposition \ref{two}(a) is equivalent to 
Proposition \ref{one}(b).
We need only note that if $k$ is odd and $A=L(a_1)U(a_2)\cdots L(a_k)$, then 
\begin{equation}
\label{odd}
[L(a_1)U(a_2)\cdots L(a_k)]^t=U(a_k)L(a_{k-1})\cdots U(a_1)= A^t
\end{equation}
which is equivalent to $(a_k,a_{k-1},\ldots , a_1)\in V_k(A^{t\prime})(\OO)=V_k(A^{\ast})(\OO)$
by Proposition \ref{one}(d).  Likewise if $k$ is even and $A=L(a_1)U(a_2)\cdots U(a_k)$, then
$L(a_k)U(a_{k-1})\cdots U(a_1)=A^t$, equivalently, $(a_k,a_{k-1},\ldots , a_1)\in V_k(A^t)$.
\end{proof}

Henceforth we restrict our attention to $V_k(A)$; in light of
Proposition \ref{one} and its proof all results are easily transferred
to the matrix factorization problems associated to $V^k(A)$ and $\oV_k(A)$.
In particular, note that we could equally well put
$\V_k(A)=V_k(A)\coprod V_k(A')$ since $V^k(A)$ and $V_k(A')$ are
defined by identical equations.
The explicit equations defining $V_k(A)$ over $\OO$ can be written down
in terms of the continuant polynomials $K_n$ of Euler \cite{e}.
The $K_n$ are defined recursively by
\begin{align}
K_{-2}=1, \quad K_{-1}&=0, \quad K_{0}=1,\quad K_1(x_1)=x_1, \\
\nonumber K_n(x_1,\ldots, x_n)&= K_{n-1}(x_1,\dots , x_{n-1})x_n+K_{n-2}(x_1,\ldots, x_{n-2})\\
\nonumber \mbox{or, equivalently, } K_n(x_1, \ldots, x_n)&=x_1 K_{n-1}(x_2,\ldots, x_n) + K_{n-2}(x_3, \ldots, x_n).
\end{align}
For example,
\begin{align*}
K_2(x_1,x_2)&=x_1x_2+1,\\
K_3(x_1,x_2, x_3)&=x_1x_2x_3+x_1+x_3,\\
K_4(x_1, x_2, x_3, x_4)&=x_1x_2x_3x_4+x_1x_2+x_1x_4+x_3x_4+1,\\
K_5(x_1,x_2,x_3,x_4,x_5)&=x_1x_2x_3x_4x_5+x_1x_2x_3+x_1x_2x_5+x_1x_4x_5+x_3x_4x_5+x_1+x_3+x_5.
\end{align*}
\begin{proposition}
\label{continuant}
\textup{1. } If $k$ is odd, 
\begin{equation*}
L(x_1)U(x_2)L(x_3)\cdots L(x_k)=\begin{pmatrix}K_{k-1}(x_2,x_3, \ldots, x_k) & K_{k-2}(x_2,x_3, \ldots, x_{k-1})\\
K_k(x_1,x_2, \ldots, x_k) & K_{k-1}(x_1,x_2, \ldots, x_{k-1})\end{pmatrix}.
\end{equation*}
\textup{2. }If $k$ is even,
\begin{equation*}
L(x_1)U(x_2)L(x_3)\cdots U(x_k)=\begin{pmatrix}K_{k-2}(x_2,x_3, \ldots, x_{k-1}) & K_{k-1}(x_2,x_3, \ldots, x_k)\\
K_{k-1}(x_1,x_2, \ldots, x_{k-1}) & K_k(x_1,x_2, \ldots, x_k)\end{pmatrix}.
\end{equation*}
\end{proposition}
\begin{proof}
Proceed by induction on $k$.
\end{proof}
Hence we deduce:
\begin{proposition}
\label{equations}
\textup{1. }If $k$ is odd, the equations defining $V_k(A)$ are 
\begin{equation*}
 \begin{pmatrix} a & c \\ b& d\end{pmatrix}=\begin{pmatrix}K_{k-1}(x_2,x_3, \ldots, x_k) & K_{k-2}(x_2,x_3, \ldots, x_{k-1})\\
K_k(x_1,x_2, \ldots, x_k) & K_{k-1}(x_1,x_2, \ldots, x_{k-1})\end{pmatrix}.
\end{equation*}
\textup{2. }If $k$ is even, the equations defining $V_k(A)$ are
\begin{equation*}
 \begin{pmatrix} a & c \\ b& d\end{pmatrix}=\begin{pmatrix}K_{k-2}(x_2,x_3, \ldots, x_{k-1}) & K_{k-1}(x_2,x_3, \ldots, x_k)\\
K_{k-1}(x_1,x_2, \ldots, x_{k-1}) & K_k(x_1,x_2, \ldots, x_k)\end{pmatrix}.
\end{equation*}
\end{proposition}
\section{The geometry of 
$V_k(A)$ }
\label{geometry}
An initial observation on $V_k(A)$ is that we know its dimension; a further observation
is that $V_{k}(A)$ is rational:
\begin{proposition}
\label{dimension}
For $k\ge3$, we have $\dim V_k(A)=k-3$ unless $k=3$ and $c=0$.  It is
irreducible unless $k=4$ and $a=1$ or $k=5$ and $A=L(b)$.  In the case
$k\ge 3$ with $c\neq 0$ if $k=3$, $V_k(A)$ is also rational
\textup{(}or the union of two irreducible rational varieties\textup{)}
with a birational map $V_k(A)\to\A^{k-3}$ given by
\begin{equation}\label{map}(x_1,\ldots,x_k)\mapsto(x_4,\ldots,x_k)\end{equation} \textup{(}with the other
component given
by \begin{equation}\label{map2}(x_1,\ldots,x_k)\mapsto(x_3,x_5,\ldots,x_k)\end{equation}
in the case when $V_k(A)$ is reducible\textup{)} and inductively
fibered.
\end{proposition}
\begin{proof}
We proceed by induction on $k$.  For $k=3$ by Proposition
\ref{equations} $V_3(A)$ is given by
\begin{equation*}
 \begin{pmatrix} a & c \\ b& d\end{pmatrix}=\begin{pmatrix}x_2x_3+1 & x_2\\
x_1x_2x_3+x_1+x_3 & x_1x_2+1\end{pmatrix}.
\end{equation*}
Hence unless $c=0$ the unique solution is given by
$$(x_1,x_2,x_3)=\left(\frac{d-1}c,c,\frac{a-1}c\right).$$  Notice that
in the case when $c=0$, the variety is empty unless we also have $a=1$
in which case it's a rational curve with a birational map to $\A^1$ given by $x_3$.

Now assume the result holds for $k-1$ by induction.  Consider the
fibration $V_k(A)\to \A^1$, given by $(x_1,x_2, \ldots, x_k)\mapsto
x_k$.  Notice the fibers of this map are given by
$V_{k-1}(AL(x_k)^{-1})$ for $k$ odd and $V_{k-1}(AU(x_k)^{-1})$ for
$k$ even.  Either way the generic fibers will be generically
inductively fibered rational varieties of dimension $k-4$ by the
induction hypothesis.  Hence, $V_k(A)$ is an inductively fibered
rational variety of dimension $k-3$.

In fact the fibers will be rational and irreducible unless $k\le6$.
We now describe the cases.  If $k=4$, the fibers will be either a
point or empty unless $a=1$ in which case one of the fibers is a
rational curve. Hence if $a=1$, then $V_4(A)$ is the reducible union
of two rational curves with maps to $\A^1$ given by $x_4$ and $x_3$
respectively.

If $k=5$, the generic fiber is irreducible with birational map given
by (\ref{map}) unless $A=L(b)$.  Hence for $A\neq L(b)$, $V_5(A)$ is
either irreducible or has a component of dimension $1$ which is
impossible since it is given by $3$ equations (the fourth is redundant
since the determinant is $1$), and thus can't have any components of
dimension less than $5-3=2$.  On the other hand if $A=L(b)$ then all
the components are reducible; hence here $V_5(A)$ is the union of two
rational surfaces with birational maps given by (\ref{map}) and
(\ref{map2}).

If $k\ge 6$, the generic fiber is always irreducible; therefore, $V_6(A)$
is always irreducible with birational map given by (\ref{map}) by the
same logic as above.
\end{proof}
\section{The arithmetic of $V_k(A)$}
\label{arithmetic}
We begin with a rather remarkable identity, which can be verified simply
by multiplying matrices.
\begin{lemma}
\label{identity}
For $v\in K^\times$, set 
\begin{align}
x_1(v) &= x_1 +\frac{(1-v^{-1})x_3}{1+x_2x_3}\\
\nonumber x_2(v) &= vx_2\\
\nonumber x_3(v) &= v^{-1}x_3\\
\nonumber x_4(v)&=x_4+\frac{(1-v)x_2}{1+x_2x_3}.
\end{align}
Then we have an equality in $\Mat_{2\times2}(K(x_1,x_2,x_3,x_4))$\textup{:}
\begin{equation}
\label{equality}
L(x_1)U(x_2)L(x_3)U(x_4)=L(x_1(v))U(x_2(v))L(x_3(v))U(x_4(v)).
\end{equation}
\end{lemma}
Suppose $P=(x_1, x_2,x_3,x_4)\in V_{4}(A)(\OO)$ with 
$A=\left(\begin{smallmatrix}a&c\\b&d\end{smallmatrix}\right)\in\SLT(\OO)$
as in
\eqref{weather}.
Then by Proposition \ref{equations}(2), we have
$a=K_2(x_2,x_3)=1+x_2x_3$.  Hence if 
$v\equiv 1 \bmod a$, then by \eqref{equality} of Lemma \ref{identity}
\begin{equation}
\label{rodent}
v.P:=(x_1(v),x_2(v),x_3(v),x_4(v))\in V_4(A)(\OO).
\end{equation}
\begin{proposition}
\label{dense}
Suppose $A\in\SLT(\OO)$ and suppose $\OO^\times$ is infinite.  If
$V_4(A)(\OO)\neq \emptyset$, then the $\OO$-points are Zariski dense.
\end{proposition}
\begin{proof}
First assume $a\ne0$.  Set $\UU_{1}(a)=\{v\in\OO^{\times}\mid v\equiv
1 \bmod a\}$. Since $\OO^\times$ is infinite by assumption,
$\UU_1(a)\subseteq \OO^\times$ is infinite too.  If $P\in
V_4(A)(\OO)$, then $v.P$ for $v\in \UU_{1}(a)$ as in \eqref{rodent} is
an infinite set of points in $V_4(A)(\OO)$.

Now suppose $a=0$, then $x_3=-x_2^{-1}$ and $$A=\begin{pmatrix} 0 & x_2
\\ x_3& x_1x_2 + x_3x_4 + 1\end{pmatrix}.$$  Hence for any $u\in\OO$,
we get $(x_1+ux_3,x_2,x_3,x_4-ux_2)\in V_4(A)(\OO)$.  Thus we again
get an infinite set of points in $V_4(A)(\OO)$.

In the case when $V_4(A)$ is irreducible this implies that the
$\OO$-points are Zariski dense.  By Proposition \ref{dimension} this
leaves us with the case when $a=1$.  In that case we get $x_2x_3=0$ so
the two components correspond to $x_2=0$ and $x_3=0$.  For $u\in\OO$,
the formulas $(u,0,b-u,c)$ and $(b,c-u,0,u)$ give infinitely many
points on each of the two components.
\end{proof}
\begin{theorem}
\label{dense1}
Suppose $A\in\SLT(\OO)$ and suppose the group of units $\OO^\times$ is infinite.\\
\textup{1)}   If the $\OO$-points of $V_k(A)$ are Zariski dense for some $k\geq 4$, then 
the $\OO$-points of $V_{k'}(A)$ are Zariski dense for all $k'\geq k$.\\
\textup{2)} 
If $k\geq 4$ and 
$V_k(A)(\OO)\neq \emptyset$, then the $\OO$-points are Zariski dense.
\end{theorem}
\begin{proof}
We start by proving (2).  We proceed by induction.  The $k=4$ case is
Lemma \ref{identity}.  Now consider the subvariety of $W\subset
V_k(A)$ given by fixing $(x_1,\ldots,x_{k-4})$.  Note that $W$ is
given by $V_4(B)$ for some matrix $B$ and is nonempty since
$(x_{k-3},\ldots,x_k)\in W$.  Thus by Lemma \ref{identity} the
$\OO$-points are Zariski dense in $W$ and in particular the images
of the $\OO$-points under the map to $\A^1$ given by $x_k$ from
Proposition \ref{dimension} are Zariski dense.  Also, by the induction
hypothesis the $\OO$-points of $V_k(A)$ are Zariski dense in each
of the fibers above each of these points in $\A^1$.  Hence, the
$\OO$-points are Zariski dense in $V_k(A)$.

Now (1) follows trivially from (2) and the fact that if $V_k(A)(\OO)\neq\emptyset$,
then $V_{k'}(A)(\OO)\neq\emptyset$ for $k'\geq k$ since
$(x_1,\ldots,x_{k})\in V_{k}(A)$ implies 
$(x_1,\ldots,x_{k},0,\ldots,0)\in V_{k'}(A)$.
\end{proof}
\begin{remark1}
\label{gone}
{\rm
There has been recent work on factoring matrices in
$\SLT(\OO)$ into products of elementary matrices.
Recall that the statement that $A\in \SLT(\OO)$
can be written as the product of $k$ elementary matrices is
equivalent to $\V_k(A)(\OO)\neq\emptyset$. However, the constructions in
\cite{mrs} and \cite{v} all write any $A\in\SLT(\OO)$ as a product of elementary
matrices beginning with a lower triangular matrix -- so in these cases we
can deduce that $V_k(A)(\OO)\neq \emptyset$.  The construction of \cite[Proof of Theorem 2.6]{jz}
writes any $A\in\SLT(\OO)$ as a product of $k$ elementary matrices beginning with an upper
triangular matrix if $k$ is odd and a lower triangular matrix if $k$ is even.
Hence in either case here we can also conclude that $V_{k}(A)(\OO)\neq\emptyset$.

Below we summarize known results:\\[.1in]
1)  For all $A\in\SLT(\OO)$, $V_{k}(A)(\OO)\neq \emptyset$ if $\OO^{\times}$ is infinite
and $k\geq 9$ (\cite[Remark (1), page 1969]{mrs}).\\
2)  Suppose $\OO^{\times}$ is infinite and $S$ contains a real archimedean prime or a finite prime.
Then for all $A\in \SLT(\OO)$ and for all  $k\geq 8$ we have $V_{k}(A)(\OO)\neq \emptyset$ (\cite[Theorem 1.2 and proof]{jz}).\\
3)  Assume the Generalized Riemann Hypothesis and suppose $\OO^{\times}$ is infinite.
Then for all $A\in\SLT(\OO)$ we have $V_{k}(A)(\OO)\neq \emptyset$ for $k\geq 5$ if $S$ contains a real
archimedean prime, for $k\geq 6$ if $S$ contains a finite prime, and for $k\geq 7$ in general (\cite[Theorem 1.3 and proof]{jz}).\\
4) Suppose $p$ is a prime.  Then for all $A\in\SLT(\Z[1/p])$ we have $V_{k}(A)(\Z[1/p])\neq \emptyset$ for $k\geq 5$
(\cite[Remark 1.1]{v}).\\
5)  There is a prime $p$ and matrix $A\in\SLT(\Z[1/p])$ such that $V_4(A)(\Z[1/p])=\emptyset$ (\cite[Proposition 5.1]{mrs}).
}
\end{remark1}
\begin{proof}[Proof of Theorem \ref{dense2}]
Combining these known results with Proposition \ref{one} and Theorem \ref{dense1} then proves our
main Theorem \ref{dense2}.
\end{proof}



\bibliographystyle{plain}
\bibliography{MFf}

\end{document}